\documentclass[11pt]{article}
\usepackage{amsbsy,amssymb,amscd,amsfonts,latexsym,amstext,delarray,amsmath,graphicx,wrapfig,float,inputenc,amsthm,subcaption,tabularx,array, lipsum,subcaption,enumitem}
\usepackage{geometry}
\newtheorem{theorem}{Theorem}[section]
\newtheorem{lemma}[theorem]{Lemma}
\newtheorem{corollary}[theorem]{Corollary}

\newtheorem{prop}[theorem]{Proposition}
\newtheorem{remark}[theorem]{Remark}

\newcounter{questions}
\newtheorem{question}[questions]{Question}

\newcommand{\parag}[1]{\vspace{2mm}

\noindent{\bf #1} }

\newcommand{\ZZ}{\ensuremath{\mathbb Z}}
\newcommand{\RR}{\ensuremath{\mathbb R}}
\newcommand{\CC}{\ensuremath{\mathbb C}}
\newcommand{\FF}{\ensuremath{\mathbb F}}

\newcommand{\VV}{\mathbf{V}}
\newcommand{\vv}{\vec{v}}

\newcommand{\pts}{\mathcal P}
\newcommand{\qts}{\mathcal Q}

\newcommand{\circs}{\mathcal C}

\newcommand{\lines}{\mathcal L}

\newcommand{\curves}{\Gamma}

\DeclareMathOperator{\cross}{cr}

\newcommand{\Szekely}{Sz\'{e}kely}

\title{On Bipartite Distinct Distances in the Plane\thanks{This research was done as part of the 2019 CUNY Combinatorics REU, supported by NSF awards DMS-1802059 and DMS-1851420.}}
\author{Surya Mathialagan\thanks{Divison of Computing \& Mathematical Sciences, California Institute of Technology, Pasadena, CA 91125. \href{mailto:surya.math@caltech.edu}{surya.math@caltech.edu}. Supported by Caltech's Summer Undergraduate Research Fellowships (SURF) Program and the Olga Taussky-Todd Award.} }
\newcommand{\enumsep}{-1mm}
\begin{document}

\maketitle

\begin{abstract}
	Given sets $\pts, \qts \subseteq \RR^2$ of sizes $m$ and $n$ respectively, we are interested in the number of distinct distances spanned by $\pts \times \qts$. Let $D(m, n)$ denote the minimum number of distances determined by sets in $\RR^2$ of sizes $m$ and $n$ respectively, where $m \leq n$. Elekes \cite{CircleGrids} showed that $D(m, n) = O(\sqrt{mn})$ when $m \leq n^{1/3}$. For $m \geq n^{1/3}$, we have the upper bound $D(m, n) = O(n/\sqrt{\log n})$ as in the classical distinct distances problem. 
	
In this work, we show that Elekes' construction is tight by deriving the lower bound of $D(m, n) = \Omega(\sqrt{mn})$ when $m \leq n^{1/3}$. This is done by adapting \Szekely's crossing number argument.  We also extend the Guth and Katz analysis for the classical distinct distances problem to show a lower bound of $D(m, n) = \Omega(\sqrt{mn}/\log n)$ when $m \geq n^{1/3}$.
\end{abstract}

\section{Introduction}

	Given a set $\pts \subseteq \RR^2$ of $n$ points, let $D(\pts)$ denote the number of distinct distances spanned by pairs of points from $\pts$. We define $D(n) = \min_{|\pts| = n} D(\pts)$, i.e., the minimum number of distinct distances determined by $n$ points in $\RR^2$. In his celebrated paper, Erd\"{o}s \cite{ErdosOriginal} showed that a $\sqrt{n} \times \sqrt{n}$ section of the integer lattice $\ZZ^2$ (see Figure \ref{fig:lattice}) determines $\Theta(n/\sqrt{\log n})$ distances.
	
	\begin{figure}[ht]
	    \centering
	    \begin{subfigure}[b]{0.4\textwidth}
	        \includegraphics{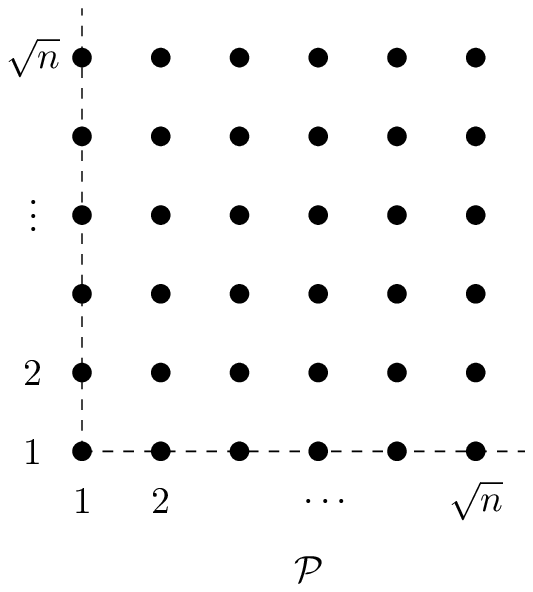}
    	    \caption{}\label{fig:lattice}
	    \end{subfigure}
	    \hspace{1cm}
	    \begin{subfigure}[b]{0.4\textwidth}
	        \includegraphics{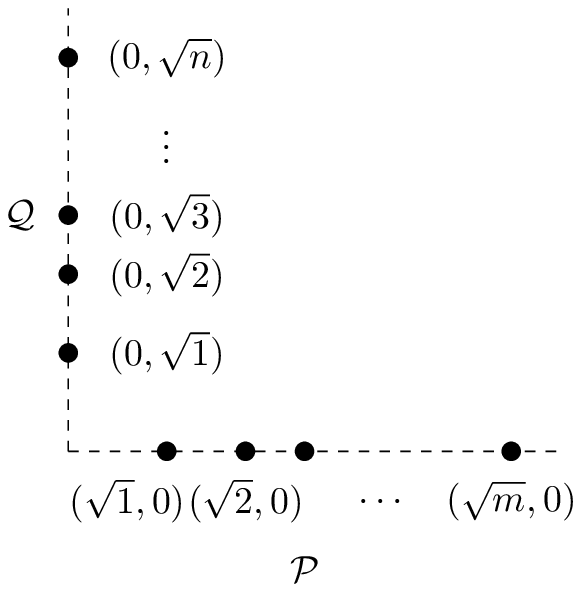}
    	    \caption{}\label{fig:orthogonal_example}
	    \end{subfigure}
	    \caption{(A) $\sqrt{n} \times \sqrt{n}$ section of $\ZZ^2$ spanning $\Theta(n/\sqrt{\log n})$ distances. (B) $\pts = \{(\sqrt{i}, 0): 1 \leq i \leq m\}$ and $\qts = \{(0, \sqrt{j}) : 1 \leq j \leq n\}$. Although $\pts$ and $\qts$ span many distances, $\pts \times \qts$ span few distances.}
	\end{figure}
	
	\begin{theorem}[\cite{ErdosOriginal}]
	    $D(n) = O(n/\sqrt{\log n})$.
	\end{theorem}
	
	Erd\"{o}s conjectured that this was asymptotically tight. Although the problem is simple to state, Erd\"{o}s was only able to show a lower bound of $\Omega(n^{1/2})$. This was followed by a series of improvements over the years (for examples, see \cite{Moser, Chung, Chungetal}). \Szekely \ \cite{Szekely} and later Solymosi and T\'{o}th \cite{SolymosiToth} used a graph-theoretic approach to improve the lower bound to $\Omega(n^{4/5})$ and $\Omega(n^{6/7})$, respectively. Later, Katz and Tardos \cite{Tardos, KatzTardos} refined their arguments to show a bound of $\Omega(n^{0.8641})$. 
	
	After over 65 years, Guth and Katz \cite{GuthKatz} showed the following almost matching lower bound for $D(n)$, resolving the problem up to a factor of $\sqrt{\log n}$.
	\begin{theorem}[\cite{GuthKatz}]\label{thm:guth_katz}
		$D(n) = \Omega(n/\log n)$. 
	\end{theorem}
	To derive Theorem \ref{thm:guth_katz}, Guth and Katz used the framework proposed by Elekes and Sharir \cite{ElekesSharir} which reduces the distinct distances problem to that of counting pairwise intersections of lines in $\RR^3$. Guth and Katz also developed several sophisticated techniques relying on tools from algebraic geometry and analytic geometry to fully resolve the problem. 
	
	While the problem of finding the asymptotic value of $D(n)$ is nearly settled, many variants of the distinct distances problem are widely open. For examples, see the survey \cite{Survey}. In the current work, we consider a bipartite variant of this problem first proposed by Elekes \cite{CircleGrids}. Given sets $\pts, \qts \subseteq \RR^2$ with $m$ and $n$ points respectively, let $D(\pts, \qts)$ denote the distinct distances spanned only by $\pts \times \qts$. That is, we ignore distinct distances spanned by pairs of points from the same set. Without loss of generality, we assume that $m \leq n$.
	
	The bipartite problem behaves quite differently from the classical variant. As an example, consider the point sets shown in Figure \ref{fig:orthogonal_example}. Although $D(\pts) = \Theta(m^2)$ and $D(\qts) = \Theta(n^2)$, we have that $D(\pts, \qts) = \Theta(m + n)$.
		
	We denote $$D(m, n) = \min_{|\pts| = m, |\qts| = n} D(\pts, \qts),$$ i.e. the minimum number of distances between two point sets. Elekes \cite{CircleGrids} showed that $D(m, n) = O(\sqrt{mn})$ for $2 \leq m \leq n^{1/3}$ using a ``circle grid'' construction (described in detail in Section \ref{sec:circle_grids}). We also have the straightforward upper bound of $D(m, n) \leq D(m + n) = O(n/\sqrt{\log n})$. However, as discussed in \cite{Survey}, it is not known if Elekes' construction is tight, and the Guth and Katz analysis for the classical problem does not seem to readily extend to this case. 
    
    In the current work, we modify and extend the Elekes, Sharir, Guth, Katz (ESGK) framework to the bipartite problem to obtain the following lower bound.
    
    \begin{theorem}\label{main_thm:guth_katz}
    For $n^{1/3} \leq m \leq n$, we have that $D(m, n) = \Omega(\sqrt{mn}/\log n).$
    \end{theorem}
    Our modifications lie within the ESGK reduction, after which the results on line incidences from \cite{GuthKatz} immediately apply. We also survey known properties of such lines. The heart of our new analysis lies in Section \ref{subsec:reguli}, where we analyse and explicitly characterise reguli that contain many lines. While doing so, we also point out a minor technical mistake in the Guth-Katz analysis of reguli (see Remark \ref{rem:guth_katz_mistake}). In Proposition \ref{prop:zariski_open}, we show that it is not difficult to resolve this mistake. Note that when $m = n$, we obtain the same lower bound as in the classical problem, which is tight up to a $\sqrt{\log n}$ factor. 
    
    Interestingly, for $m \leq n^{1/3}$, we obtain a tight bound without relying on the tools introduced by Guth and Katz. Instead, we adapt an older argument by \Szekely \ \cite{Szekely} to get rid of the logarithmic factor. \Szekely \ used a graph-theoretic approach, and more specifically, the crossing lemma to introduce shorter and elegant proofs for various problems in combinatorial geometry. In particular, he derived the improved lower bound at the time of $D(n) = \Omega(n^{4/5})$ for the classical distance problem \cite{Szekely}. Adapting this proof for the bipartite variant, we get the following bound. 
    
    \begin{theorem}\label{main_thm:szekely}
        For $2 \leq m \leq n^{1/3}$, we have that $D(m, n) = \Omega(\sqrt{mn}).$
    \end{theorem}

    This shows that Elekes' construction is indeed tight, completely resolving the bipartite distinct distances problem in this range.
    
    Table \ref{table:summary} summarises the current results for $D(m, n)$, for each range of $m$. All of the lower bounds are from this work.
    \begin{table}[ht]
    \centering 
    {\renewcommand{\arraystretch}{1.2}
        \begin{tabular}{c|c|c}
            \hline 
            Range of $m$ & Lower Bound & Upper Bound\\
            \hline 
            \hline 
            $m = 1$ & 1 & 1\\
            $2 \leq m \leq n^{1/3}$ & $\Omega(\sqrt{mn})$ & $O(\sqrt{mn})$\\
            $n^{1/3} \leq m \leq n^{1/2}/\log^{1/4} n$ & $\Omega(\sqrt{mn}/\log n)$ & $O(m^2)$\\
            $n^{1/2}/\log^{1/4} n \leq m \leq n$ & $\Omega(\sqrt{mn}/\log n)$ & $O(n/\sqrt{\log n})$\\
            \hline 
        \end{tabular}
    }
    \caption{Bounds on $D(m, n)$ for different ranges of $m$. For the upper bounds, see Corollary \ref{corr:upper_bounds}. The lower bounds come from Theorem \ref{main_thm:guth_katz} and Theorem \ref{main_thm:szekely}.} \label{table:summary}
    \end{table}
    
    From Table \ref{table:summary}, we see that there is still a gap between the upper and lower bounds when $n^{1/3} \leq m \leq n$. 
    
    \begin{question}
        What is the asymptotic value of $D(m, n)$ when $n^{1/3} \leq m \leq n$?
    \end{question}
    
    Improving the bound when $m = n$ would also eliminate the gap for $D(n)$. As shown by Guth and Katz \cite{GuthKatz}, such an improvement cannot be obtained from the ESGK framework without making significant changes.
    
    While $D(m, n) = \Theta(\sqrt{mn})$ for $m \leq n^{1/3}$, we have that $D(m, n) = o(\sqrt{mn})$ when $m = n$. Therefore, somewhere in the range $n^{1/3} \leq m \leq n$, we are able to achieve a better asymptotic bound. It would be interesting to find the smallest $m$ for which $D(m, n) = o(\sqrt{mn})$. 
    
    It can be shown that if all the points in $\pts$ lie on a line, then $D(\pts, \qts) = \Omega(\sqrt{mn})$. See Remark \ref{rem:szekely_set_line} for a crossing-based proof, and see Exercise 3.4 in \cite{AdamTextbook} for an incidence proof. Therefore, any construction with $D(\pts, \qts) = o(\sqrt{mn})$ would have to be structurally different from Elekes' construction. 

    \parag{Outline.} In Section \ref{sec:circle_grids}, we restate Elekes' construction for bipartite distinct distances, and present the best known upper bounds on $D(m, n)$ for all $1 \leq m \leq n$. In Section \ref{sec:crossing_number}, we will introduce \Szekely's crossing number approach and prove Theorem \ref{main_thm:szekely}. We then present the modified ESGK reduction for the bipartite problem in Section \ref{sec:ESGK} and outline the proof of Theorem \ref{main_thm:guth_katz}. We analyse the structure of the lines obtained through the reduction in Section \ref{sec:planes_reguli}, and complete the proof of Theorem \ref{main_thm:guth_katz}.
    
    \parag{Notation.} For points $p, q \in \RR^d$, we will denote the distance between them as $d(p, q)$. For sets of points $\pts, \qts \subset \RR^d$, we define $$d(\pts, \qts) = \min_{p \in \pts, q \in \qts} d(p, q).$$
    We will say that $A \lesssim B$ (respectively, $A \gtrsim B)$ if there exists some constant $c > 0$ such that $A \leq cB$ (respectively, $A \geq cB$). We use $O_{v_1, v_2, \dots, v_k}$ to represent the usual big-$O$ notation where the constant of proportionality depends on the variables $v_1, \dots, v_k$. We define $\Omega_{v_1, v_2, \dots, v_k}$ and $\Theta_{v_1, v_2, \dots, v_k}$ symmetrically. 
\section{Elekes' Circle Grid Construction}\label{sec:circle_grids}

In this section, we explicitly restate Elekes' circle grid construction. We present it in a simplified manner and extend it to the case where $m \geq n^{1/3}$. 

\parag{Construction.} Suppose $2 \leq m \leq n^{1/3}$ and set $s = \sqrt{n/m}$. Consider the following sets of points: 
    \begin{equation}\label{eq:elekes_construction}
        \begin{split}
            \pts &= \{(a, 0) : 1 \leq a \leq m\},\\
            \qts &= \{(i, \sqrt{j}) : 1 \leq i \leq s, s^2 + 1 - i^2 \leq j \leq s^2 + ms - i^2\}. 
        \end{split}
    \end{equation}
    In Elekes' formulation, $\qts$ is viewed as the intersections of $m$ vertical lines and $\sqrt{mn}$ circles centered at each point in $\pts$. As noted in \cite{ResearchProblems}, this arrangement of points can be viewed as the ``hyperbolic image'' of the usual $n$ lattice points and $m\sqrt{mn}$ straight lines with many incidences, embedded in the Poincar\'{e} model of the upper half plane.
       
    \begin{prop}\label{prop:circle_grid_analysis}
        For the sets defined in \eqref{eq:elekes_construction}, we have $D(\pts, \qts) = \Theta(\sqrt{mn})$. 
    \end{prop}
    \begin{proof}
        Since $m \leq n^{1/3}$, note that $s = \sqrt{n/m} \geq m$. The square of the distance between the points $(i, \sqrt{j})$ and $(a, 0)$ is an integer of the form $(a - i)^2 + j$. To give an upper bound on the number of distinct distances in $\pts \times \qts$, it suffices to check the maximal and minimal values of $(a-i)^2 + j$.
        \begin{align*}
            (a - i)^2 + j &\geq a^2 - 2ai + i^2 + s^2 + 1 - i^2 = a^2 - 2ai + s^2 + 1\\
            &\geq a^2 - 2as + s^2 + 1 \geq m^2 - 2ms + s^2 + 1.
        \end{align*}
        The last inequality follows from the fact that $a^2 - 2as$ is minimised when $a = m$ (recall that $m \leq s$). Similarly,
        \begin{align*}
            (a - i)^2 + j &\leq a^2 - 2ai + i^2 + s^2 + ms - i^2 = a^2 - 2ai + s^2 + ms\\
            &\leq a^2 - 2a + s^2 + ms \leq m^2 - 2m + s^2 + ms.
        \end{align*}
        Therefore, the number of distinct distances is at most $$D(\pts, \qts) \leq (m^2 - 2m + s^2 + ms) - (m^2 - 2ms + s^2 + 1) + 1 = 3ms = O(\sqrt{mn}).$$
        Moreover, the point $(1, 0)$ has $ms = \sqrt{mn}$ distances with points of the form $(1, \sqrt{j}) \in \qts$. Therefore, $D(\pts, \qts) \geq \sqrt{mn}$. Combining these upper and lower bounds, we have the desired result. 
    \end{proof}
    
    \begin{remark}\label{remark:circle_grid} Although the construction is still well defined when $m > n^{1/3}$, one of the main steps in the above analysis requires $m \leq s$. A similar analysis shows that the construction spans $\Theta(m^2)$ distances when $m > n^{1/3}$.  
    \end{remark}
    
    Now, we state the current best upper bounds for this problem in various ranges. 
    \begin{corollary}\label{corr:upper_bounds}
        For $2 \leq m \leq n$, we have that
        $$ 
        D(m, n) = 
        \begin{cases}
            1 & m = 1\\
            O(\sqrt{mn}) & 2 \leq m \leq n^{1/3}\\
            O(m^2) & n^{1/3} \leq m \leq n^{1/2}/\log^{1/4} n\\
            O(n/\sqrt{\log n}) & n^{1/2}/\log^{1/4} n \leq m \leq n. 
        \end{cases}
        $$
    \end{corollary}
    \begin{proof}
        $D(1, n) = 1$ since one can place all the points in $\qts$ on a circle centered at the point in $\pts$. The bound in the range $2 \leq m \leq n^{1/3}$ follows from Proposition \ref{prop:circle_grid_analysis}. If $m \geq n^{1/3}$, by Proposition \ref{prop:circle_grid_analysis}, $D(m, n) \leq D(m, m^3) = O(m^2).$ Finally, we also have the straightforward bound: $D(m, n) \leq D(m + n) = O(n/\sqrt{\log n}).$
    \end{proof}
    
\section{Crossing Number Arguments}\label{sec:crossing_number}

    \Szekely \ \cite{Szekely} showed that many results in discrete geometry could be obtained using crossing numbers. For example, he presented short proofs for the point-line incidence problem and the unit distances problem. In particular, Sz\'{e}kely showed the following bound which led to the best known lower bound at that time for the classical distinct distances problem.
    
    \begin{theorem}[\cite{Szekely}] \label{thm:szekely_classical}
        For any set of $n$ points $\pts$, there exists $p \in \pts$ that determines $\Omega(n^{4/5})$ with $\pts\setminus \{p\}$.
    \end{theorem}
    
    In this section, we adapt Sz\'{e}kely's approach to derive a tight bound for $D(m, n)$ when $2 \leq m \leq n^{1/3}$. We first introduce two main tools required for this proof: crossing number bounds for multigraphs and the Szemer\'{e}di-Trotter bound for the number of point-line incidences. 
    
    \parag{The Crossing Lemma.} In a \emph{drawing} of a graph, every vertex is a distinct point in the plane and every edge is a Jordan arc connecting the two corresponding vertices. We assume that the interior of every such arc does not contain vertices, that any two arcs have a finite number of intersections, and that no three arcs intersect at the same point.  We define the \textit{crossing number} $\cross(G)$ of a graph $G$ to be the minimum number of edge crossings across all drawings of $G$.
    
    Sz\'{e}kely's main tool was the following asymptotically tight lower bound for the crossing number of a graph. This was shown independently by Ajtai et al \cite{CrossingLemma} and Leighton \cite{CrossingLemma2}. 
    
    \begin{theorem}[\cite{CrossingLemma}, \cite{CrossingLemma2}]\label{thm:crossing_lemma}
        For a simple graph $G = (V, E)$ where $|V| = n$ and $|E| = e$ with $e \geq 4n$, we have that $$\cross(G) = \Omega \left(\frac{e^3}{n^2}\right).$$ 
    \end{theorem}
    
    For the purpose of Theorem \ref{thm:szekely_classical}, Sz\'{e}kely introduced an analogue of Theorem \ref{thm:crossing_lemma} for multigraphs (graphs which can have parallel edges --- multiple edges between the same pair of vertices). 
    
    \begin{theorem}[\cite{Szekely}]\label{thm:multigraph_crossing_lemma}
        For a multigraph $G = (V, E)$ where $|V| = n$, $|E| = e$ and maximum edge multiplicity $m$, if $e > 5mn$, then $$\cross(G) = \Omega\left(\frac{e^3}{mn^2}\right).$$
    \end{theorem}
    
    \parag{The Szemer\'{e}di-Trotter Theorem.} Given a set $\pts$ of $m$ points and a set $\lines$ of $n$ lines, both in $\RR^2$, an incidence is a pair $(p, \ell) \in \pts \times \lines$ such that $p \in \ell$. Erd\"{o}s and Purdy \cite{ErdosPurdy} constructed $\pts$ and $\lines$ with $\Theta(m^{2/3}n^{2/3} + m + n)$ incidences, and conjectured that this is asymptotically optimal. This conjecture was resolved by Szemer\'{e}di and Trotter.
    
    \begin{theorem}[Szemer\'{e}di-Trotter Theorem \cite{SzeTrotter}]\label{thm:szemeredi_trotter}
        Let $\pts$ be a set of $m$ points and $\lines$ be a set of $n$ lines. The number of incidences in $\pts \times \lines$ is $O(m^{2/3}n^{2/3} + m + n)$.
    \end{theorem}
    
    Theorem \ref{thm:szemeredi_trotter} was one of the main theorems \Szekely \ was able to reprove using an elegant crossing number argument. 
    
    Given a set $\pts$ and an integer $r \geq 2$, we say that a line $\ell$ is $r$\textit{-rich} if $\ell$ contains at least $r$ points from $\pts$. The Szemer\'{e}di-Trotter theorem gives immediately implies a bound on the number of $r$-rich lines.
    
    \begin{theorem}[\cite{Szekely, InclusionExclusion}]\label{thm:rich_line_bound}
        Let $\pts$ be a set of $m$ points and let $r \geq 2$. Then, the number of $r$-rich lines is $$O\left(\frac{m^2}{r^3} + \frac{m}{r}\right).$$
    \end{theorem}
    
    We use Theorem \ref{thm:rich_line_bound} to analyse rich \textit{perpendicular bisectors} of pairs of points in $\RR^2$. The bisector of $p$ and $q$ is the set of all points that are equidistant to $p$ and $q$. 
    
    \parag{Proof of Theorem \ref{main_thm:szekely}.} To prove Theorem \ref{main_thm:szekely}, we derive the following stronger statement. 
    
    \begin{theorem}\label{thm:szekely_bipartite}
        Consider a set $\pts$ of $m$ points and a set $\qts$ of $n$ points, with $2 \leq m \leq n^{1/3}$. Then there exists a point in $\pts$ that determines $\Omega(\sqrt{mn})$ distances with the points in $\qts$.
    \end{theorem}
    
    \begin{proof}
        Let $t = \max_{p \in \pts} D(p, \qts)$ and assume that $t = o(\sqrt{mn})$ (otherwise we are done). For each point $p \in \pts$, draw at most $t$ concentric circle centered at $p$ so that each circle contains at least one point from $\qts$, and every point in $\qts$ is contained in some circle. Denote the set of resulting circles as $\circs$ and note that $|\circs| \leq mt$. We construct a corresponding topological multigraph $G = (V, E)$ as follows:
        
            \begin{enumerate}[itemsep=\enumsep]
                \item Set $|V| = n$ so that each vertex in $V$ corresponds to a point in $\qts$.   
                \item For every circle in $\circs$, for each arc between consecutive points in this circle, add an edge to the graph between the corresponding vertices.
                \item Delete edges corresponding to circles that contain at most two points.
            \end{enumerate}
        Note that a circle incident to $k$ points in $\qts$ leads to $k$ edges in Step (2). Thus, for each point $p \in \pts$, we constructed $n$ edges in the graph. After this step, $|E| = mn$. For each vertex $p \in \pts$, since there are at most $t$ circles corresponding to it, we are deleting at most $2mt = o(mn)$ edges at Step (3). We conclude that $|E| = \Theta(mn)$ after Step (3). 
        
        Consider the drawing of the graph $G$ with vertices corresponding to the points in $\qts$, and edges corresponding to the arcs of the circles in $\circs$ (with slight perturbations to avoid more than three concurrent edges). Since $|\circs| = O(mt)$ and every two circles intersect twice, we conclude that $\cross(G) \lesssim m^2t^2$. 
        
        To apply Theorem \ref{thm:multigraph_crossing_lemma}, we need an upper bound for the maximum edge multiplicity. While $G$ might have high edge multiplicity, we can delete edges to reduce the multiplicity without changing the asymptotic size of $|E|$. This is stated more precisely in the following proposition, which we prove after the current proof.
        
        \begin{prop}\label{prop:bisector_bound}
            For an integer $r \geq 2$, let $T$ be a set of pairs $(\ell, e)$ such that $e = (u, v) \in E$, the line $\ell$ is the perpendicular bisector of $u$ and $v$, and $\ell$ is incident to at least $r$ points of $\pts$. If $e$ and $e'$ are parallel edges, then $(\ell, e)$ and $(\ell, e')$ represent two distinct pairs in $T$. Then, $$|T| = O\left(\frac{tm^2}{r^2} + t m \log m \right).$$
        \end{prop}
        
        Note that if vertices $u$ and $v$ have more than $r$ edges between them, then they are consecutive on more than $r$ circles of $\circs$. This in turn implies that the perpendicular bisector of $u$ and $v$ is $r$-rich. Therefore, we can use Proposition \ref{prop:bisector_bound} for some constant $r = K$ to bound the number of edges with multiplicity at least $K$. This number is
            \begin{align*}  
                \frac{ctm^2}{K^2} + ct m \log m &\lesssim \frac{ \sqrt{mn} \cdot m^2}{K^2} \lesssim \frac{m \cdot \sqrt{m^3n}}{K^2} \lesssim \frac{mn}{K^2}
            \end{align*}
        where the last inequality follows from the fact that $m^3 \leq n$. For a sufficiently large constant $K$, we delete at most half the edges in $G$. Denote the resulting subgraph as $G'$. Then, $G'$ has $n$ vertices, $\Theta(mn)$ edges and maximum edge multiplicity $K$. Applying Theorem \ref{thm:multigraph_crossing_lemma}, we have that $$m^2t^2 \gtrsim \cross(G) \geq \cross (G') \gtrsim \frac{e^3}{Kn^2} \gtrsim \frac{m^3n^3}{n^2} = m^3n.$$  Rearranging the inequality, we get the desired bound of $t \gtrsim \sqrt{mn}$. 
    \end{proof}

    \begin{remark}
        The assumption $m = O(n^{1/3})$ is crucial in the proof. If $m = \omega(n^{1/3})$, we have to delete edges with multiplicity larger than $K\sqrt{mt/n}$ edges for some large constant $K$. Then, the above proof gives a weaker bound of $D(m, n) = \Omega(m^{3/5}n^{1/5})$. When $m = n$, we recover Theorem \ref{thm:szekely_classical}.   
    \end{remark}
    
    Although Proposition \ref{prop:bisector_bound} is a bipartite analogue of the bisector bound in \cite{Szekely}, the proof generalises immediately. We include it here for completeness. 
    
    \begin{proof}[Proof of Proposition \ref{prop:bisector_bound}]
        Let $\ell$ be a line incident to $k$ points of $\pts$. These points correspond to the centers of at most $kt$ circles in $\circs$, and each such circle contains at most two arcs $e$ such that $(\ell, e) \in T$. Hence, $\ell$ participates in at most $2kt$ pairs in $T$. 
        
        By Theorem \ref{thm:rich_line_bound}, we know that for $2^i \leq \sqrt{m}$, there are at most $cm^2/2^{3i}$ perpendicular bisectors that are $2^i$-rich. By a dyadic decomposition, the number of pairs in $T$ corresponding to $r$-rich edges for $2^i \leq \sqrt{m}$ is at most $$\sum_{i: r \leq 2^i \leq \sqrt{m}} 2t \cdot 2^{i+1} \cdot \frac{cm^2}{(2^i)^3} = 4ctm^2 \sum_{i: r \leq 2^i \leq \sqrt{m}}\frac{1}{2^{2i}} \leq 4ctm^2 \cdot  \frac{c'}{r^2} = O\left(\frac{tm^2}{r^2}\right).$$
        When $2^i \geq \sqrt{m}$, we have at most $cm/2^i$ perpendicular bisectors that are $2^i$-rich. A similar dyadic decomposition argument implies that the number of pairs of $T$ in this case is at most
        $$\sum_{i: \sqrt{m} \leq 2^i \leq m} 2t \cdot 2^{i+1} \cdot \frac{cm}{2^i} = \sum_{i: \sqrt{m} \leq 2^i \leq m} 4ctm \leq 4ctm \log m = O(tm \log m).$$
        Combining the above bounds leads to the result.
    \end{proof}
    
    As an immediate consequence of Theorem \ref{thm:szekely_bipartite}, we have the bound of $D(m, n) = \Omega(\sqrt{mn})$ when $m \leq n^{1/3}$ as asserted in Theorem \ref{main_thm:szekely}. 
    
    \begin{remark} \label{rem:szekely_set_line}
    Suppose $\pts$ is a set of $m$ points on a line $\ell$, where $2 \leq m \leq n$. Then, $\ell$ is the only 2-rich line, and it corresponds to the centers of all the circles in $\circs$. Therefore, the size of $T$ as defined in Proposition \ref{prop:bisector_bound} is bounded by $2kt$. After deleting all the edges with multiplicity at least 2, we obtain a simple graph with $\Theta(mn)$ edges. Following the rest of the argument in Theorem \ref{thm:szekely_bipartite}, we have $D(\pts, \qts) = \Omega(\sqrt{mn})$ for every $m$.      
    \end{remark}
    
    \begin{remark}
        While Theorem \ref{thm:szekely_bipartite} implies that there exists one point in $\pts$ that spans $\Omega(\sqrt{mn})$ distances, in Elekes' construction, every point $p \in \pts$ spans $\Theta(\sqrt{mn})$ distances. 
    \end{remark}

\section{Modified ESGK Reduction} \label{sec:ESGK}
To prove their distinct distances theorem, Guth and Katz \cite{GuthKatz} adapted a preceding reduction by Elekes and Sharir \cite{ElekesSharir}.
In this section, we modify this reduction for the bipartite variant. In Section \ref{sec:planes_reguli}, we complete the analysis by proving lemmas that were used in the current section. The heart of our new analysis lies in Section \ref{subsec:reguli}, where we further develop our knowledge about lines contained in a regulus. At the end of this section, we list the places where our reduction is different from the original one.

\parag{Bipartite Distance Energy.} Let $\pts$ be a set of $m$ points and $\qts$ be a set of $n$ points such that $2 \leq m \leq n$. We define the \textit{bipartite distance energy} of $\pts \times \qts$ to be the set of quadruples: $$E(\pts, \qts) = \{(p_1, q_1, p_2, q_2)\ |\ p_i \in \pts, q_i \in \qts, d(p_1, q_1) = d(p_2, q_2) \neq 0\}.$$
We now use a standard Cauchy-Schwarz argument to relate the bipartite distance energy to $D(\pts, \qts)$.

\begin{prop}\label{prop:cauchy_schwarz}
    For any set $\pts$ of $m$ points and set $\qts$ of $n$ points, we have $$D(\pts, \qts) \geq \frac{m^2n^2}{|E(\pts, \qts)|}.$$
\end{prop}

\begin{proof}   
    Let $\delta_1, \delta_2, \dots, \delta_x$ be the distinct distances between $\pts$ and $\qts$, and let $d_i$ be the number of pairs $(p, q) \in \pts \times \qts$ at distance $\delta_i$. Since each of the $mn$ pairs in $\pts \times \qts$ contributes to exactly one $\delta_i$, we get $\sum d_i = mn$. Applying the Cauchy-Schwarz inequality, we have that $$|E(\pts, \qts)| = \sum_{i=1}^x d_i^2 \geq \frac{
1}{x} \left(\sum_{i=1}^x d_i \right)^2 = \frac{m^2n^2}{D(\pts, \qts)}.$$ Rearranging, we have the desired inequality.
\end{proof}

By Proposition \ref{prop:cauchy_schwarz}, to show that $D(\pts,\qts) = \Omega(\sqrt{mn}/\log n)$, it suffices to show that $|E(\pts, \qts)| = O(m^{3/2}n^{3/2} \log n)$.  

\parag{Rigid Motions.} A transformation of $\RR^2$ is a \textit{rigid motion} if it preserves distances between points, and it is a \textit{proper rigid motion} if it also preserves orientation. Let $G$ denote the group of proper rigid motions of the plane. It is well known that $G$ consists of translations and rotations of $\RR^2$.

\begin{prop}\label{prop:unique_rigid_motion}
    For each $(p_1, q_1, p_2, q_2) \in E(\pts, \qts)$, there exists a unique $g \in G$ so that $g(p_1) = q_2$ and $g(q_1) = p_2$.
\end{prop}

\begin{proof}
    All proper rigid motions $g$ taking $p_1$ to $q_2$ can be obtained by first translating the plane by $q_2 - p_1$, and then applying a rotation around $q_2$. Since $|p_1 - q_1| = |q_2 - p_2| \neq 0$, exactly one such rotation also takes $q_1$ to $p_2$.
\end{proof}

    Using Proposition \ref{prop:unique_rigid_motion}, we obtain a map $\varphi: E(\pts, \qts) \to G$ which associates each quadruple $(p_1, q_1, p_2, q_2) \in E(\pts, \qts)$ with the unique $g \in G$ that satisfies $g(p_1) = q_2$ and $g(q_1) = p_2$. 

    We can write $G$ as a disjoint union $G^{trans} \cup G^{rot}$, where $G^{trans}$ is the set of all translations and $G^{rot}$ is the set of rotations of $\RR^2$. Consider the subset $E^{trans}(\pts, \qts) \subseteq E(\pts, \qts)$ of quadruples that are mapped to a rigid motion in $G^{trans}$, and let $E'(\pts, \qts) = E(\pts, \qts) \setminus E^{trans}(\pts, \qts).$ 
    
    \begin{prop}\label{prop:bound_translations}
        $|E^{trans}(\pts, \qts)| = O(m^2n)$.
    \end{prop}
    \begin{proof}
    Given $p_1, p_2 \in \pts$, $q_1 \in \qts$, there is exactly one translation $g$ mapping $g(q_1) = p_2$, and therefore there is at most one $q_2 \in \qts$ such that $g(p_1) = q_2$. The number of choices for $p_1, p_2 \in \pts$ and $q_1 \in \qts$ is $m^2n$, and each choice could be completed to at most one quadruple in $E^{trans}(\pts, \qts)$. 
    \end{proof} 
    
    To prove Theorem \ref{main_thm:guth_katz}, it suffices to show that $E'(\pts, \qts) = O(m^{3/2}n^{3/2} \log n)$. 
    
    \parag{Reduction to Lines in $\RR^3$.} Any rotation $g \in G^{rot}$ fixes some point $(o_x, o_y)$ and rotates around this point with some counterclockwise angle $0 < \alpha < 2\pi$. We define the map $\rho: G^{rot} \to \RR^3$ such that 
    \begin{equation}\label{eqn:parametrisation}
        \rho(g) = (o_x, o_y, \cot(\alpha/2)),
    \end{equation}
    a simplication introduced by Guth and Katz. Under this parametrisation, one can check that the set of all rotations taking a point $p \in \RR^2$ to $q \in \RR^2$ form the line \begin{equation}\label{eqn:rotation_line}
        \ell_{p,q} = \left\{ \left( \frac{p_x + q_x}{2}, \frac{p_y +q_y}{2}, 0\right) + t \left(\frac{q_y - p_y}{2}, \frac{p_x - q_x}{2}, 1\right): t \in \RR\right\}. 
    \end{equation}
    Note that there exists a rotation taking $p_1$ to $q_2$ and $q_1$ to $p_2$ if and only if the lines $\ell_{p_1, q_2}$ and $\ell_{q_1, p_2}$ intersect. Indeed, the point of intersection is the parameterisation of the rotation $\varphi(p_1, q_1, p_2, q_2)$. Therefore, a quadruple is in $E'(\pts, \qts)$ if and only if $\ell_{p_1, q_2}$ and $\ell_{q_1, p_2}$ intersect.
    
    Let
   \begin{equation}\label{eqn:ESGK_lines_defn}
       \lines^1 = \{\ell_{p, q}\}_{p \in \pts, q \in \qts}, \quad \lines^2 = \{\ell_{q, p}\}_{p \in \pts, q \in \qts} \quad \mbox{and }\lines = \lines^1 \cup \lines^2.
   \end{equation}
    Let $I(\lines^1, \lines^2)$ denote the number of pairs of intersecting lines in $\lines^1 \times \lines^2$. The above bijection gives us that 
    \begin{equation}\label{eqn:quad_to_incidence}
        |E'(\pts, \qts)| = I(\lines^1, \lines^2).
    \end{equation}
     We have now reduced the problem of bounding the cardinality of the bipartite distance energy to the problem of bounding the number of pairs of intersecting lines. More specifically, to showing that $I(\lines^1, \lines^2) = O(m^{3/2}n^{3/2}\log n)$. Note that $|\lines^i| = mn$ and $|\lines| = 2mn$. 
    
    \begin{remark} 
    It may seem tempting to associate the lines in $\lines^1$ to pairs of points in $\pts \times \pts$ and those in $\lines^2$ to pairs of points in $\qts \times \qts$. However, this approach leads to a much more difficult problem. In particular, having $2mn$ lines rather than $m^2 + n^2$ lines seems to be crucial in our proof.
    \end{remark}
    
    \parag{Rich points.} Consider a point $p$ in $\RR^3$ that is incident to $r$ lines of $\lines$. If $x$ of those lines are associated with $\lines^1$, then the number of pairs in $\lines^1 \times \lines^2$ that intersect at $p$ is $x(r-x) \leq r^2/2$. We call a point $r$\textit{-rich} if it is incident to at least $r$ lines in $\lines$, and let $m_r(\lines)$ denote the number of $r$-rich points. Then, we have 
    \begin{equation}\label{eqn:ESGK_simplified_bound}
        I(\lines^1, \lines^2) \lesssim \sum_{r = 2}^{2mn} r^2 (m_{r}(\lines) - m_{r+1}(\lines)).
    \end{equation}
    Therefore, it suffices to bound $m_r(\lines)$. 
    
    In general, $2mn$ lines can have a lot more than $m^{3/2}n^{3/2}\log n$ pairs of intersecting lines. When many lines lie on a common plane or regulus, or if many lines intersect at a point, we can have up to $\sim m^2n^2$ pairs of intersecting lines. To overcome these issues, we rely on the following results due to Guth and Katz \cite{GuthKatz}. 
    
    \begin{theorem}[\cite{GuthKatz}]\label{thm:rich_2}
        Suppose $\lines$ is a set of $N$ lines. If no more than $\sqrt{N}$ lines of $\lines$ lie on any plane and $O(\sqrt{N})$ lines of $\lines$ lie on any common regulus, we have $m_2(\lines) = O(N^{3/2})$.
    \end{theorem}
    
    \begin{theorem}[\cite{GuthKatz}]\label{thm:rich_gtr_3}
        Suppose $\lines$ is a set of $N$ lines. If no more than $\sqrt{N}$ lines of $\lines$ lie on any plane, we have $m_r(\lines) = O(N^{3/2}/r^2)$ for all $3 \leq r \leq N^{1/2}$.
    \end{theorem}
    
    In Section \ref{sec:planes_reguli}, we will prove the following two lemmas.

    \begin{lemma}\label{lem:plane_bound_main}
    Consider $\lines$ as defined in \eqref{eqn:ESGK_lines_defn}. The following hold.
    \begin{enumerate}[label=(\roman*), itemsep = \enumsep]
        \item Every point of $\RR^3$ is incident to at most $2m$ lines of $\lines$. 
        \item Every plane in $\RR^3$ contains at most $2m$ lines of $\lines$.
    \end{enumerate}
    \end{lemma}

    \begin{lemma}\label{lem:regulus_bound_main}
    For $\lines$ as defined in \eqref{eqn:ESGK_lines_defn}, at least one of the following holds.
    \begin{enumerate}[label=(\roman*), itemsep = \enumsep]
        \item $D(\pts, \qts) = \Omega(\sqrt{mn})$.
        \item Every regulus in $\RR^3$ contains $O(\sqrt{mn})$ lines of $\lines$.
    \end{enumerate}
    \end{lemma}
    
    We are now ready to prove Theorem \ref{main_thm:guth_katz}. 
    
    \begingroup
    \def\thetheorem{\ref{main_thm:guth_katz}}
    \begin{theorem}
        For $n^{1/3} \leq m \leq n$, we have that $D(m, n) = \Omega(\sqrt{mn}/\log n).$
    \end{theorem}
    \addtocounter{theorem}{-1}
    \endgroup

    \begin{proof}
        Suppose that $D(\pts, \qts) = O(\sqrt{mn})$ (otherwise, we are done). By Lemma \ref{lem:plane_bound_main} and Lemma \ref{lem:regulus_bound_main}, at most $2m$ lines in $\lines$ lie in a common plane and $O(\sqrt{mn})$ lines in $\lines$ lie in a common regulus. Lemma \ref{lem:plane_bound_main} also gives us that $\pts_r(\lines) = 0$ for $r > 2m$ since no point is incident to more than $2m$ lines. Combining \eqref{eqn:ESGK_simplified_bound} with Theorems \ref{thm:rich_2} and \ref{thm:rich_gtr_3}, we have:
        \begin{align*}
            I(\lines^1, \lines^2) &\lesssim \sum_{r = 2}^{2mn} r^2(m_r(\lines) - m_{r+1}(\lines))\\
            &\lesssim m_2(\lines) + \sum_{r = 3}^{2m} (r^2 - (r-1)^2) m_r(\lines)\\
            &\lesssim O(m^{3/2}n^{3/2}) + \sum_{r = 3}^{2m} (2r - 1) \cdot O\left(\frac{m^{3/2}n^{3/2}}{r^2}\right) \lesssim m^{3/2}n^{3/2} \log n.
        \end{align*} 
        Combining this with with Proposition \ref{prop:bound_translations} and \eqref{eqn:quad_to_incidence}, $$|E(\pts, \qts)| = |E^{trans}(\pts, \qts)| + |E'(\pts, \qts)| = O(m^2n) + O(m^{3/2}n^{3/2} \log n) = O(m^{3/2}n^{3/2} \log n).$$
        Finally, by Proposition \ref{prop:cauchy_schwarz}, we have $$D(\pts, \qts) \geq \frac{m^2n^2}{|E(\pts, \qts)|} = \Omega\left(\frac{\sqrt{mn}}{\log n}\right),$$ concluding the proof.
    \end{proof}

    \parag{Main modifications.} Here are the main modifications that we made to adapt the reduction for the bipartite problem. 
    \begin{enumerate}[label = (\roman*), itemsep = \enumsep]
        \item We consider the bipartite distance energy rather than the standard distance energy.
        \item We consider the set of rigid motions taking a pair of points in $\pts \times \qts$ to a pair of points in $\qts \times \pts$.
        \item We present a somewhat different analysis in Section \ref{sec:planes_reguli} to bound the number of lines in any or regulus.
    \end{enumerate}
\section{Lines, Planes, and Reguli} \label{sec:planes_reguli}
In this section, we prove Lemmas \ref{lem:plane_bound_main} and \ref{lem:regulus_bound_main}. This is the last remaining piece in our proof of Theorem \ref{main_thm:guth_katz}. 

Throughout this section, we identify each point in $\RR^3$ with the rotation described by the map $\rho$ in \eqref{eqn:parametrisation}. In Section \ref{subsec:lines}, we state some properties of lines in $\RR^3$. Then, in Section \ref{subsec:planes}, we bound the number of lines in $\lines$ that lie on any given plane in $\RR^3$. We then introduce some definitions and tools from algebraic geometry in Section \ref{subsec:alg_geom}. Finally, in Section \ref{subsec:reguli}, we define reguli, state properties of reguli and bound the number of lines in $\lines$ that lie on any given regulus. 

For the rest of this section, for each $p \in \pts$, define $\lines_p^1 = \{\ell_{p,q}|q \in \qts\}$, $\lines_p^2 = \{\ell_{q,p}\ |\ q \in \qts\}$, $\curves_p^1 = \{\ell_{p, x}| x \in \RR^2\}$, and $\curves_p^2 = \{\ell_{x, p}| x \in \RR^2\}$.

\subsection{Lines}\label{subsec:lines}
In this section, we present some properties of lines in $\RR^3$. Following the notation from \cite{GuthBook}, we call a line in $\RR^3$ \textit{horizontal} if it has a constant $z$-coordinate.

\begin{prop}\label{prop:line_properties}
We have the following properties of lines in $\RR^3$.
    \begin{enumerate}[label=(\roman*), itemsep = \enumsep]
        \item For all $p, q \in \RR^2$, the line $\ell_{p, q}$ is not horizontal, and hence it intersects the $xy$-plane.
        \item Every non-horizontal line in $\RR^3$ is of the form $\ell_{p, q}$ for some unique $p, q \in \RR^2$.
        \item For each $p \in \RR^2$ and $P \in \RR^3$, there exist $q, r \in \RR^2$ such that $P \in \ell_{p, q}$ and $P \in \ell_{r, p}$.
        \item For all $p, q, r \in \RR^2$ where $q \neq r$, the lines $\ell_{p, q}$ and $\ell_{p, r}$ are skew. The same holds for $\ell_{q, p}$ and $\ell_{r, p}$.
        \item For all $p, q \in \RR^2$, the line $\ell_{p, q}$ is a reflection of $\ell_{q, p}$ across the $xy$-plane.
    \end{enumerate}
\end{prop}

\begin{proof}
    These observations follow immediately from \eqref{eqn:rotation_line} and the fact that every point represents a rotation of $\RR^2$.
    \begin{enumerate}[label=(\roman*), itemsep = \enumsep]
        \item From \eqref{eqn:rotation_line}, we see that the direction of $\ell_{p, q}$ has a non-zero $z$-coordinate, implying that it is not horizontal.
        \item Any non-horizontal line $\ell$ intersects the $xy$-plane at some point $(a, b, 0)$, and has direction $(d, e, 1)$. In other words, it can be written uniquely in the form $$\ell = \left\{(a, b, 0) + t(d, e, 1)\ |\ t \in \RR^2\right\}.$$
        Equating this to \eqref{eqn:rotation_line}, we obtain a system of four linearly independent linear equations in $p_x, p_y, q_x, q_y$. This system always has a unique solution. 
        \item The point $P \in \RR^3$ represents some rotation $g$. Set $q = g(p)$ and $r = g^{-1}(p)$. Then, $P \in \ell_{p, q}$ and $P \in \ell_{q, r}$.
        \item Since a rigid motion is a bijection from $\RR^2$ to $\RR^2$, no $g \in G'$ can take $p$ to both $q$ and $r$. Therefore, $\ell_{p, q}$ and $\ell_{p, r}$ cannot intersect. From \eqref{eqn:rotation_line}, these lines cannot be parallel. Therefore, these two lines are skew. A symmetric argument shows that $\ell_{q, p}$ and $\ell_{r, p}$ are skew.
        \item Consider a rotation $g$. Note that $g$ and $g^{-1}$ have the same fixed point. This implies that the corresponding points in $\RR^3$ have the same $x$ and $y$ coordinates. Moreover, if $g$ is a rotation with angle $\theta$, then $g^{-1}$ is a rotation with angle $2\pi - \theta$. Note that the cotangent function is an odd function with period $\pi$. In other words, $\cot(\theta/2) = -\cot((2\pi - \theta)/2)$, so the $z$ coordinates of the corresponding points in $\RR^3$ are negations of each other. Therefore, the points corresponding to $g$ and $g^{-1}$ are reflections of each other across the $xy$-plane. Since $g \in \ell_{p, q} \iff g^{-1} \in \ell_{q, p}$, we conclude that $\ell_{p, q}$ is the reflection of $\ell_{q, p}$ across the $xy$-plane.
    \end{enumerate}
\end{proof}
We now study the geometric interpretation of horizontal lines in $\RR^3$ under $\rho$. 
We define an \textit{oriented line} to be a line with an associated directional vector with the same slope as the line. 
We say that two oriented lines are \textit{parallel} if their vectors have the same direction and \textit{anti-parallel} if they have opposite directions. Two oriented lines \textit{subtend} an angle $\theta$ if this is the counterclockwise angle subtended by two directional vectors corresponding to the line orientations. If the oriented lines are parallel, then $\theta = 0$. If the oriented lines are anti-parallel, then $\theta = \pi$.

In the rest of this section, we use $\ell$ to denote lines in $\RR^3$ and $\lambda$ to denote lines in $\RR^2$. For oriented non-parallel lines $\lambda_1, \lambda_2 \subset \RR^2$, let $S(\lambda_1, \lambda_2)$ be the set of points in $\RR^3$ corresponding to rotations of $\RR^2$ that map $\lambda_1$ onto $\lambda_2$ while preserving the orientation. This notation and the following proposition are based on ideas from \cite{GuthBook}.

\begin{prop}\label{prop:ESGK_horizontal} For oriented non-parallel lines $\lambda_1, \lambda_2 \subset \RR^2$, the set $S(\lambda_1, \lambda_2)$ is a horizontal line in $\RR^3$. Moreover, for any horizontal line $\ell \subset \RR^3$ there exist (non-unique) oriented non-parallel lines $\lambda_1, \lambda_2$ so that $\ell = S(\lambda_1, \lambda_2)$. 
\end{prop}
    
\begin{figure}[ht]
        \centering
        \includegraphics{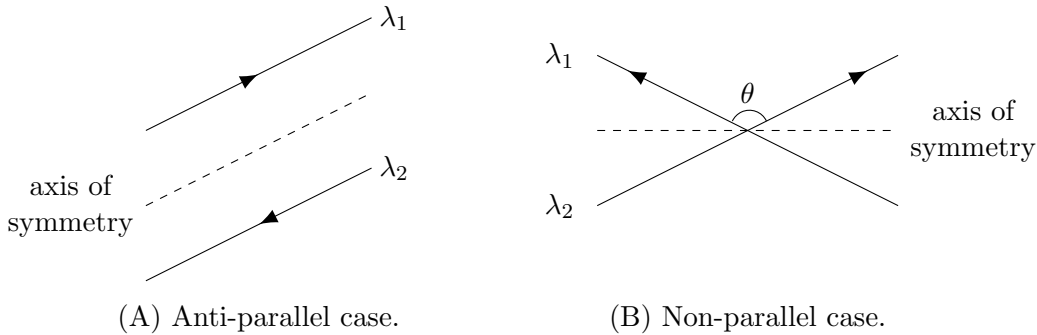}
        \caption{The axis of symmetry corresponding to all $O$ that are the centers of a rotation taking $\lambda_1$ to $\lambda_2$.}
        \label{fig:line_rigid_motion}
\end{figure}

\begin{proof}
    Let $\theta$ be the angle subtended by $\lambda_1$ and $\lambda_2$. Note that rigid motions maintain the equivalence classes of parallel lines. Suppose $g \in S(\lambda_1, \lambda_2)$ fixes some point $O = (o_x, o_y)$.  
    For $i \in \{0, 1\}$, denote by $\lambda_i'$ the line through $O$ that is parallel to $\lambda_i$. Then, $g$ maps $\lambda_1'$ to $\lambda_2'$ through a rotation around $O$. This implies that the angle of rotation is $\theta$. Therefore, all $g \in S(\lambda_1, \lambda_2)$ are rotations of angle $\theta$, so $S(\lambda_1, \lambda_2)$ is contained in the plane $z = \cot(\theta/2)$. 
    
    It remains to show that the set of fixed points of all $g \in S(\lambda_1, \lambda_2)$ is form a line in $\RR^2$. Since $g(\lambda_1) = \lambda_2$ and $g(\lambda_1') = \lambda_2'$, we get that $d(\lambda_1, \lambda_1') = d(\lambda_2, \lambda_2')$. Therefore, $O$ lies on an axis of symmetry of $\lambda_1$ and $\lambda_2$ (that is, a line consisting of the points that are equidistant from $\lambda_1$ and $\lambda_2$). If $\lambda_1$ and $\lambda_2$ are anti-parallel, there is only one axis of symmetry. Otherwise, there are two axes of symmetry. Fix a rotation $g \in S(\lambda_1, \lambda_2)$. From the above, we know that $g$ is a rotation of angle $\theta$ that fixes a point $O \in \RR^2$. For $i \in \{1, 2\}$, let $P_i$ denote the foot of the perpendiculars from $O$ to $\lambda_i$. Note that $P_i$ is the unique point on $\lambda_i$ closest to $O$, and $$d(O, \lambda_1) = d(O, P_1) = d(O, P_2) = d(O, \lambda_2).$$ Since $g(P_1) \in \lambda_2$, we get that $g(P_1) = P_2$. In particular, the angle subtended by $\overrightarrow{OP_1}$ and $\overrightarrow{OP_2}$ is equal to $\theta$. This happens only when $O$ lies on the axis of symmetry that bisects the angle not subtended by $\lambda_1$ and $\lambda_2$. That is, if $\lambda_1$ and $\lambda_2$ have directions $\vv_1$ and $\vv_2$ respectively, this is the axis of symmetry parallel to $\vv_1 - \vv_2$. These axes of symmetry are depicted in Figure \ref{fig:line_rigid_motion}. Moreover, for any $O$ on this chosen axis of symmetry, the rotation $g$ around $O$ such that $g(P_1) = P_2$ lies in $S(\lambda_1, \lambda_2)$. This completes the proof of the first part of the proposition. 
    
    Consider a horizontal line $\ell$ in $\RR^3$ lying on the plane $z = w$. Let $\lambda$ be the projection of $\ell$ on the $xy$-plane. Let $\theta = 2\cot^{-1} (w)$. If $\theta = \pi$, let $\lambda_1$ and $\lambda_2$ be two distinct parallel lines at distance 1 from $\lambda$ with opposite orientations. Otherwise, pick an arbitrary point $p$ on $\lambda$. Pick $\lambda_1$ and $\lambda_2$ to be two oriented lines through $p$ subtending an angle of $\theta$ such that $\lambda$ is the axis of symmetry that does not bisect the angle subtended by $\lambda_1$ and $\lambda_2$. Then, $S(\lambda_1, \lambda_2) = \ell$, as desired.
\end{proof}

\subsection{Planes}\label{subsec:planes}

Using Proposition 5.1, we can bound the number of lines in any plane and complete the proof of Lemma \ref{lem:plane_bound_main}. We restate the lemma for convenience.

\begingroup
    \def\thetheorem{\ref{lem:plane_bound_main}}
    \begin{lemma}
        Consider $\lines$ as defined in \eqref{eqn:ESGK_lines_defn}. The following hold.
        \begin{enumerate}[label=(\roman*), itemsep = \enumsep]
            \item Every point of $\RR^3$ is incident to at most $2m$ lines of $\lines$. 
            \item Every plane in $\RR^3$ contains at most $2m$ lines of $\lines$.
        \end{enumerate}
    \end{lemma}
    \addtocounter{theorem}{-1}
    \endgroup
    
\begin{proof}
By Proposition \ref{prop:line_properties}, any two lines in  $\lines_p^i$ are skew. Therefore, no two lines from the same family $\lines_p^i$ are incident to the same point in $\RR^3$ or lie on the same plane. Since we have $2m$ such sets, we obtain both parts of the lemma.
\end{proof}

\subsection{Algebraic Geometry Preliminaries}\label{subsec:alg_geom}
Before studying reguli, we introduce some basic algebraic geometry. For more information, see for example \cite{AlgGeom, Whitney}.

In the following, we will take $\FF$ to be either $\RR$ or $\CC$. Given polynomials $f_1, \dots, f_k \in \FF[x_1, \dots, x_d]$, the \textit{affine variety} $\VV(f_1, \dots, f_k)$ is defined as $$\VV(f_1, \dots, f_k) = \{(a_1, \dots, a_d) \in \FF^d \ |\ f_i(a_1, \dots, a_d) = 0, 1 \leq i \leq k\}.$$
If $U$ is a variety and $U' \subseteq U$ is a variety, we say $U'$ is a \textit{subvariety} of $U$. If we can write $U = V \cup W$ where $V$ and $W$ are proper subvarieties, we say $U$ is \textit{reducible}. Otherwise, we say $U$ is \textit{irreducible.} 

We now state basic properties of varieties without proof. 

\begin{theorem}[Special case of Hilbert's basis theorem] Every variety in $\RR^d$ can be described by a single polynomial.
\end{theorem} 

\begin{prop}
    Let $U$ and $W$ be varieties in $\FF^d$. Then, $U \cap W$ and $U \cup W$ are both varieties. 
\end{prop}

\parag{Degree and Dimension.} There are several non-equivalent definitions of degree of a variety in $\RR^d$. For our purposes, we define the \textit{degree} of a variety $U \in \RR^d$ as 
$$\min_{\substack{f_1, \dots, f_k \in \RR[x_1, \dots, x_d]\\ \VV(f_1, \dots, f_k) = U}} \max_{1 \leq i \leq k} \deg(f_i).$$
The \textit{dimension} $d'$ of an irreducible variety $U \subseteq \RR^d$ is the maximum integer for which there exists a sequence: $$U_0 \subset U_1 \subset \dots \subset U_{d'} = U$$ where all the varieties are irreducible and all the containments are proper. The dimension of a reducible variety $U$ is the maximum dimension of an irreducible component of $U$. As an example, a zero-dimensional variety is a finite set of points, and a one-dimensional variety is a finite union of curves and points. 

We rely on the following results about intersections of varieties.

\begin{theorem}[B\'{e}zout's theorem]\label{thm:bezout}
    Consider $f, g \in \FF[x, y]$. If $f$ and $g$ do not have any common factors, $\VV(f) \cap \VV(g)$ consists of at most $\deg(f) \cdot \deg(g)$ points.
\end{theorem}

\begin{theorem}[\cite{Joints}]\label{thm:r3_bezout}
    Consider $f, g \in \FF[x, y, z]$. If $f$ and $g$ do not have any common factors, $\VV(f) \cap \VV(g)$ contains at most $\deg(f) \cdot \deg(g)$ lines. 
\end{theorem}

\begin{theorem}\label{thm:irred_components}
    Let $f \in \FF[x_1, x_2, \dots, x_d]$ be a polynomial of degree $D$. Then, the number of irreducible components of $\VV(f)$ is $O_{d, D}(1)$.
\end{theorem}

We also require the following distinct distances bound, which is an application of B\'{e}zout's theorem.

\begin{lemma}\label{lem:curve_bipartite_bound}
    Consider finite sets $A, B \subset \RR^2$, where $|A| \geq 2$ and $|B| = x$. Suppose that all the points in $B$ lie on a one-dimensional algebraic variety $\gamma$ of degree $D$. Then, $D(A, B) = \Omega_D(x)$. 
\end{lemma}

\begin{proof}
    By Theorem \ref{thm:irred_components}, the variety $\gamma$ has $O_D(1)$ irreducible components. By the pigeonhole principle, there exists a one-dimensional irreducible component $\gamma'$ of $\gamma$ that contains $\Omega_D(n)$ points of $B$. Pick some $a \in A$. If $\gamma'$ is a circle, choose an $a$ that is not the center of $\gamma'$. Consider the set of circles $\circs$ centered at $a$ and containing at least one point from $B$. For every circle $c \in \circs$, since $c$ and $\gamma'$ are distinct irreducible curves, they do not have a common component. By Theorem \ref{thm:bezout}, we have that $|c \cap \gamma'| \leq 2D$. Since $\circs$ covers all the points in $B$, we have that $|\circs| = \Omega_D(x)$. Therefore, $D(A, B) \geq D(a, B) = \Omega_D(x)$.
\end{proof}

\parag{Zariski Topology.} The \textit{Zariski topology} on a variety $U \subseteq \FF^d$ is the topology where the closed sets are algebraic varieties in $\FF^d$. Thus, an open set is $U \setminus W$ for a variety $W \subset U$. If $X \subset \FF^d$, the \textit{Zariski closure} $\overline{X}$ is the smallest variety in $\FF^d$ that contains $U$. In particular, $X$ is Zariski open in $\overline{X}$ if $\overline{X} \setminus X$ is a variety.

\parag{Complexification.} Given a variety $U \subset \RR^d$, the \textit{complexification} $U^* \subset \CC^d$ is the smallest complex variety that contains $U$. Every complex variety that contains $U$ also contains $U^*$. Such a complexification always exists, and $U$ is precisely the set of real points in $U^*$ \cite{Whitney}. 
We will use $\Re(V)$ to denote the set of all real points of a complex variety $V$. The dimension of a complex variety is defined in the same way as the dimension of a real variety. For the degree of a complex variety, see for example \cite[Definition 18.1]{Complexity}. We only require the standard property that a real variety $U$ has degree $O(1)$ if and only if its complexification $U^*$ has degree $O(1)$.

\parag{Constructible Sets and Projections.}
A set $X$ is \textit{constructible} if there exist varieties $X_1, X_2, \dots, X_{\ell}$ such that $\dim X_{j+1} < \dim X_j$ for every $1 \leq j \leq \ell - 1$, and 
\begin{equation}\label{eqn:constructible}
X = \left(\left(\left(X_1 \setminus X_2\right) \cup X_3 \right) \setminus X_4 \dots \right).
\end{equation}
Note that $X$ is Zariski open in its Zariski closure $\overline{X}$. We define the \textit{complexity} of $X$ to be $\min(\deg(X_1) + \deg(X_2) + \dots + \deg(X_\ell))$ where the minimum is taken over all representations of $X$ of the form \eqref{eqn:constructible}.  This definition is not standard. However, since we are interested only in constructible sets of bounded complexity, any reasonable definition of complexity would work equally well. For further details, see for example \cite[Section 3]{Complexity}. For a constructible set $X \subseteq \CC^d$, we will denote by $\Re(X)$ the set of real points contained in $X$. Then $\Re(X)$ is a constructible set in $\RR^d$.

In both $\RR^d$ and $\CC^d$, a projection of a variety need not be a variety. For instance, if we project a circle in the $xz$-plane of $\RR^3$ onto the first two coordinates, then we obtain a line segment. In $\RR^d$, projections of constructible sets need not be constructible. However, in the case of $\CC^d$, we have the following result. 
\begin{theorem}[\cite{Complexity}]\label{thm:constructible_projection}
    Let $X \subset \CC^d$ be a constructible set of dimension $d'$ and complexity $k$. Let $\pi: \CC^d \to \CC^e$ be a projection on $e$ out of $d$ coordinates of $\CC^d$. Then, $\pi(X)$ is constructible set of dimension at most $d'$ and complexity $O_{k, d}(1)$.
\end{theorem}

\subsection{Reguli} \label{subsec:reguli}
When studying lines in reguli, while we use some tools from \cite{GuthKatz}, we present a somewhat different argument. We first define a regulus and describe some properties of reguli. We then characterise reguli that contain many lines, and provide a geometric approach to bound the number of lines in any regulus.

For three pairwise-skew lines $\ell_1, \ell_2, \ell_3$, let $\Psi(\ell_1, \ell_2, \ell_3)$ be the set of lines in $\RR^3$ that intersect all three lines.  A \textit{regulus} is the Zariski closure of $\Psi(\ell_1, \ell_2, \ell_3)$ for three pairwise-skew lines $\ell_1, \ell_2, \ell_3$. We denote such as regulus as $R(\ell_1, \ell_2, \ell_3)$.

\parag{Properties of Reguli.} It is known that all reguli are quadratic surfaces in $\RR^3$, i.e. it can be written as $R = \VV(f)$ where $f \in \RR[x, y, z]$ is of degree 2 (see for example, \cite[Section 5.2]{AdamTextbook}). If the three pairwise-skew lines $\ell_1, \ell_2, \ell_3$ lie in parallel planes, then the corresponding regulus is a hyperbolic paraboloid. Otherwise, the corresponding regulus is a hyperboloid of one sheet. 

Reguli are \textit{doubly-ruled}. That is, for every point $p$ on a regulus $R$, there exist at least two lines that are contained in $R$ and incident to $p$. The set of lines that are contained in a regulus $R$ can be partitioned into two disjoint sets, called \textit{rulings}. The lines of a ruling are pairwise-disjoint and pairwise-skew, and their union is $R$. In the regulus $R(\ell_1, \ell_2, \ell_3)$, the lines in $\Psi(\ell_1, \ell_2, \ell_3)$ lie on one ruling of the regulus, and $\ell_1, \ell_2, \ell_3$ lie on the other ruling. 

\begin{prop}\label{prop:zariski_open}
     The set $\Psi(\ell_1, \ell_2, \ell_3)$ is Zariski open in $R(\ell_1, \ell_2, \ell_3)$. That is, we can write $$R(\ell_1, \ell_2, \ell_3) = Z \cup Z_0$$ where $Z$ is the union of all lines in $\Psi(\ell_1, \ell_2, \ell_3)$, and $Z_0$ is a one-dimensional variety with degree $O(1)$.
\end{prop}

\begin{proof}
    Applying a generic isometry of $\RR^3$, we may assume that $\ell_1, \ell_2, \ell_3$ are non-horizontal and that $R(\ell_1, \ell_2, \ell_3)$ contains $O(1)$ horizontal lines. Denote $Z$ as the union of all lines in $\Psi(\ell_1, \ell_2, \ell_3)$. Since $R(\ell_1, \ell_2, \ell_3)$ is the Zariski closure of $Z$, it suffices to show that $Z$ is constructible. 
    
    Complexify the three lines $\ell_1, \ell_2, \ell_3$ to obtain lines in $\CC^3$. Abusing notation, we also refer to these complex lines as $\ell_1, \ell_2, \ell_3$. We say that a non-horizontal line $\ell$ has parametrisation $(a, b, c, d) \in \CC^4$ if we can define $\ell$ by
    $$x = az + b \quad \text{and} \quad y = cz + d.$$
    
    For $i \in \{1, 2, 3\}$, suppose $\ell_i$ has parametrisation $(a_i, b_i, c_i, d_i)$. Consider some line $\ell$ with parametrisation $(a, b, c, d)$. Note that $\ell_i$ and $\ell$ intersect if and only if there exists a $z \in \RR$ for which 
    \begin{equation}\label{eqn:line_intersection}
        a_i z + b_i = a z + b \quad \text{and} \quad c_i z + d_i = cz + d.
    \end{equation}
     Solving \eqref{eqn:line_intersection}, unless $a = a_i$ or $c = c_i$, we have 
    \begin{equation}\label{eqn:variety_defn_c4}
        (a_i - a)(d_i - d) = (c_i - c)(b_i - b).
    \end{equation}
    If at least one of $a = a_i$ and $c = c_i$ holds, then we have the following two cases.
    \begin{enumerate}[label=(\roman*), itemsep=\enumsep]
         \item If $a = a_i$ and $c = c_i$, then $\ell$ and $\ell_i$ are parallel and do not intersect. In this case, note that \eqref{eqn:variety_defn_c4} also holds.
         \item If $a = a_i$ but $c \neq c_i$, we see that  \eqref{eqn:variety_defn_c4} holds only if $b = b_i$. In this case, choosing $z$ to be the unique solution to $c_iz + d_i = cz + d$, we have that \eqref{eqn:line_intersection} is satisfied and $\ell$ and $\ell_i$ intersect. A similar argument shows that the lines intersect also in the case where $a \neq a_i$ and $c = c_i$. 
    \end{enumerate}
    Let $V$ be the set of points $(a, b, c, d) \in \CC^4$ that satisfy \eqref{eqn:variety_defn_c4} for all $i \in \{1, 2, 3\}$. Note that $V$ is a variety. Consider some non-horizontal line $\ell$ in $\CC^3$ with parametrisation $(a, b, c, d) \in V$. For each $i \in \{1, 2, 3\}$, either $\ell$ and $\ell_i$ intersect or $\ell$ and $\ell_i$ are parallel. Since $\ell_1, \ell_2, \ell_3$ are pairwise-skew, there is at most one $(a, b, c, d) \in V$ parallel to $\ell_1$ but intersecting $\ell_2$ and $\ell_3$. More generally, at most three points $(a, b, c, d) \in V$ correspond to lines that are not transversal to all three of $\ell_1, \ell_2, \ell_3$. Call this set of finite exceptions $V_0$. 
    
    Now, let $U \subseteq \CC^7$ be the set of points $(a, b, c, d, x, y, z) \in U$  that satisfy \eqref{eqn:variety_defn_c4} and that $(x, y, z) \in \CC^3$ lies on the line with parametrisation $(a, b, c, d)$. Note that $U$ is a variety. Consider $$U_0 = \{(a, b, c, d, x, y, z) \in U \ | \ (a, b, c, d) \in V_0\}.$$ 
    Since $V_0$ is a set of at most three points, $U_0$ is a set of at most three lines, so it is a one-dimensional variety. By definition, $M = U \setminus U_0$ is a constructible set. 
    
    Let $\pi: \CC^7 \to \CC^3$ be the projection on the last three coordinates. Clearly, $\pi(M)$ is the union of the lines in $\CC^3$ that are transversal to $\ell_1, \ell_2, \ell_3$. In particular, $$\pi(M) = \bigcup_{\ell \in \curves} \ell$$ where $\curves$ is the set of all lines in $\CC^3$ with parametrisation $(a, b, c, d) \in V \setminus V_0$. By Theorem \ref{thm:constructible_projection}, we have that $\pi(M)$ is constructible. In other words, we have that $\pi(M)$ is Zariski open in its closure. Therefore, we can write $\pi(M) = \overline{\pi(M)} \setminus P_0$ where $P_0$ is a variety of dimension at most one and degree $O(1)$.  
    
    The real part of any line in $\CC^3$ is either a real line or a single point. Let $S_0$ be the set of all lines in $\curves$ that contain a single real point. A simple dimension counting argument implies that, for a surface to  contain a two-dimensional family of lines, it must be infinitely ruled by those lines. Thus, any non-planar variety in $\CC^3$ contains at most a one-dimensional family of lines. In particular, $S_0$ is at most a one-dimensional set of lines. It is not difficult to then show that $\Re(S_0)$ is a semi-algebraic set of dimension at most one. Since the real part of any line in $\curves \setminus S_0$ is a real line, this implies that $\Re(\pi(M)) \setminus \Re(S_0) \subseteq \Psi(\ell_1, \ell_2, \ell_3)$. Set $$S_1 = \bigcup_{\ell \in \curves} \Re(\ell) \setminus \Re(S_0).$$ Then, $\Psi(\ell_1, \ell_2, \ell_3) \setminus S_1$ is a set of $O(1)$ horizontal lines transversal to all $\ell_1, \ell_2, \ell_3$. We denote by $H$ the union of these lines. Note that $A = (\overline{\pi(M)} \cup H^*) \setminus (P_0 \cup S_0)$ is a constructible set in  $\CC^3$. Moreover, it is easy to check that $$Z = \left(\Re(\overline{\pi(M)}) \cup H\right) \setminus (\Re(P_0) 
    \cup \Re(S_0)) = \Re(A),$$ 
    is the real part of the complex constructible set $A$, or in other words, $Z$ constructible in $\RR^3$. Hence, $Z$ is Zariski open in $R(\ell_1, \ell_2, \ell_3)$.
\end{proof}

    Combining Proposition \ref{prop:zariski_open} with the fact that a regulus can be partitioned into two distinct rulings, we have the following corollary. 

\begin{corollary}\label{corr:regulus_ruling_line}
     Any line in one ruling of the regulus intersects all but $O(1)$ lines in the other ruling of the regulus.
\end{corollary}

\begin{remark}\label{rem:guth_katz_mistake}
    Guth and Katz \cite{GuthKatz} incorrectly claim that every line in one ruling intersects every line in the other ruling of the regulus. For example, consider the hyperboloid of one-sheet $\VV(x^2 + y^2 - z^2 - 1)$. By symmetry, the tangent planes at the points $v_1= (1, 0, 0)$ and $v_2 = (-1, 0, 0)$ are parallel. For $i \in \{1, 2\}$, the two lines through $v_i$ lie on the tangent plane at $v_i$. Therefore, both lines through $v_1$ do not intersect either of the lines through $v_2$. It seems seems possible that this issue disappears when moving to projective space $\RR \mathbb{P}^3$, as it does in this particular example. We chose not to pursue this direction.
\end{remark}

\parag{Reguli and Lines.} Now, we analyse the relationship between reguli and the lines in $\RR^3$ under $\rho$. We begin with an observation due to Guth and Katz.

\begin{lemma}[\cite{GuthKatz}]\label{lem:gk_regulus}
    Suppose that a regulus $R$ contains at least seven lines of $\curves_p^i$. Then, all the lines in one ruling of $R$ lie in $\curves_p^i$. 
\end{lemma}
Recall from Proposition \ref{prop:line_properties} that the set of lines in $\curves_p^1$ is a reflection of the lines in $\curves_p^2$ across the $xy$-plane. Thus, although Guth and Katz only showed the above statement for the case of $\curves_p^1$, it applies to the case of $\curves_p^2$ as well. In the following, we derive theorems only for the case of $\curves_p^1$, and these also hold symmetrically for $\curves_p^2$.

Guth \cite{GuthBook} describes two examples of reguli where one ruling falls entirely within one family $\curves_p^i$ for some $p \in \RR^2$. Let $C(q, r)$ denote the circle centered at $q \in \RR^2$ of radius $r > 0$. We now describe our first construction.

    \begin{figure}[ht]
            \centering
            \includegraphics{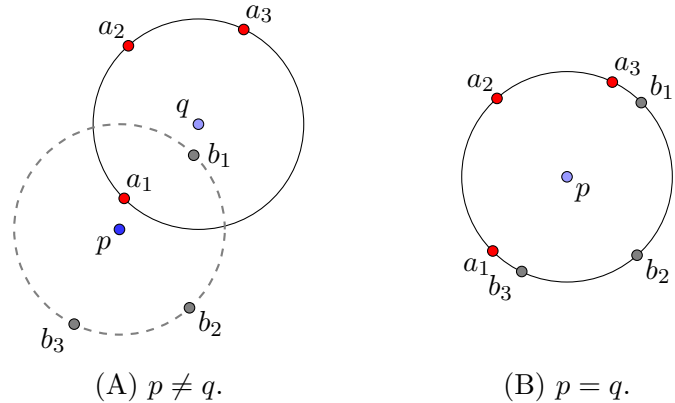}
            \caption{Layout of points satisfying $p - b_i = a_i - q$. In (B), $a', b', c'$ are diametrically opposite from $a, b, c$ respectively.}
            \label{fig:regulus_circle}
    \end{figure} 
    
\begin{prop}\label{prop:regulus_by_circle}
    Consider $p, q \in \RR^2$, and $r > 0$. There exists a regulus $R$ where one ruling consists of the lines of the form $\{\ell_{p, a}\}_{a \in C(q, r)}$, and the other ruling is all lines of the form $\{\ell_{b, q}\}_{b \in C(p, r)}$.
\end{prop}

\begin{proof}
    
    Consider $a_1, a_2, a_3 \in C(q, r)$. Let $b_i$ be the unique point such that $p - b_i = a_i - q$ (see Figure \ref{fig:regulus_circle}). Note that the lines $\ell_{p, a_i}$ and $\ell_{b_i, q}$ do not intersect since the rigid motion corresponding to $(p, b_j, q, a_i)$ corresponds to a translation. 
    
    For all $b \neq b_i$, the lines $\ell_{p, a_i}$ and $\ell_{b, q}$ intersect since $(p, b, q, a_i)$ does not correspond to a translation. Hence, the regulus generated by $\ell_{p, a_1}, \ell_{p, a_2}, \ell_{p, a_3}$ contains all lines $\ell_{b, q}$ such that $b \in C(p, r) \setminus \{b_1, b_2, b_3\}$. Denote this regulus as $R$. Now, pick points $x_1, x_2, x_3 \in C(q, r) \setminus \{a_1, a_2, a_3\}$, and let $y_i$ be the unique point such that $p - y_i = x_i - q$. By a symmetric argument, we have that the regulus generated by $\ell_{p, x_1}, \ell_{p, x_2}, \ell_{p, x_3}$ contains all lines $\ell_{b, q}$ such that $b \in C(p, r) \setminus \{y_1, y_2, y_3\}$. Since the two reguli have infinitely many lines in common, by Theorem \ref{thm:r3_bezout}, we have that $$R(\ell_{p, x_1}, \ell_{p, x_2}, \ell_{p, x_3}) = R.$$
    Hence, $\{\ell_{b, q}\}_{b \in C(p, r)}$ lies in $R$. Note that every line of $\{\ell_{p, a}\}_{a \in C(q, r)}$ intersects infinitely many lines from $\{\ell_{b, q}\}_{b \in C(p, r)}$. Thus, the lines of $\{\ell_{p, a}\}_{a \in C(q, r)}$ are also contained in $R$.
    
    By Lemma \ref{lem:gk_regulus}, we know that $\{\ell_{p, a}\}_{a \in C(q, r)}$ lie in one ruling, and $\{\ell_{b, q}\}_{b \in C(p, r)}$ lie in the other. More specifically, the first ruling lies entirely in $\curves_p^1$ and the second ruling lies entirely in $\curves_q^2$.
    
    Suppose $\ell_{p, x}$ lies on the first ruling. By Corollary \ref{corr:regulus_ruling_line}, we have that $\ell_{p, x}$ intersects all but a constant number of lines in the second ruling. Suppose $\ell_{p, x}$ intersects some $\ell_{b, q}$ for $b \in C(q, r)$. In other words, $d(x, q) = d(p, a) = r$, and we have $x \in C(q, r)$. Therefore, the first ruling is exactly $\{\ell_{p, a}\}_{a \in C(q, r)}$. Similarly, the second ruling is exactly $\{\ell_{b, q}\}_{b \in C(p, r)}$. 
\end{proof}

\begin{remark}
    Since $\ell_{p, a}$ intersects the $xy$-plane at the point $\left(\frac{p_x + a_x}{2}, \frac{p_y + a_y}{2}, 0\right)$, it is easy to verify that the intersection of $$\bigcup_{a \in C(q, r)}\ell_{p, a}$$ with the $xy$-plane is a circle. Since the cross-section of a hyperbolic paraboloid cannot be a closed curve, the regulus described above is a hyperboloid of one sheet.
\end{remark}

For the second construction, recall from Proposition \ref{prop:ESGK_horizontal} that every horizontal line is of the form $S(\lambda_1, \lambda_2)$ for two non-parallel oriented  lines in $\RR^2$. 
        
        \begin{figure}[ht]
            \centering
            \includegraphics{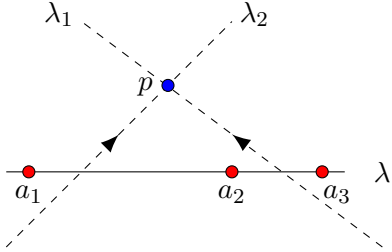}
            \caption{Example of two oriented lines $\lambda_1$ and $\lambda_2$ through $p$.}
            \label{fig:regulus_collinear}
        \end{figure}

\begin{prop}\label{prop:regulus_by_line}
     Consider some $p \in \RR^2$ and some oriented line $\lambda \subset \RR^2$. There exists a regulus where one ruling consists of lines of the form $\{\ell_{p, a}\}_{a \in \lambda}$, and the other ruling is all the lines of the form $$\{S(\lambda', \lambda)\ |\ \mbox{$\lambda'$ is an oriented line containing $p$, not parallel to $\lambda$}\}.$$
     By symmetry, there also exists a regulus where one ruling consists of lines of the form $\{\ell_{b, p}\}_{b \in \lambda}$, and the other ruling consists of the lines of the form $$\{S(\lambda, \lambda')\ |\ \lambda' \mbox{ is an oriented line containing $p$, not parallel to $\lambda$}\}.$$ 
\end{prop}

\begin{proof} 
    It suffices to show the statement for the first construction since the second construction is a reflection of the first across the $xy$-plane. 

    Arbitrarily choose three distinct points $a_1, a_2, a_3 \in \lambda$, as shown in Figure \ref{fig:regulus_collinear}. By Proposition \ref{prop:line_properties}, the lines $\ell_{p, a_1}, \ell_{p, a_2}, \ell_{p, a_3}$ are pairwise-skew. We denote as $R$ the regulus $R(\ell_{p, a_1}, \ell_{p, a_2}, \ell_{p, a_3})$.
   
    First, we show that $\Psi(\ell_{p,a_1}, \ell_{p,a_2}, \ell_{p,a_3})$ does not contain non-horizontal lines. Consider some line $\ell_{x, y}$. If a line $\ell_{x, y}$ is transversal to the three lines $\ell_{p, a_i}$, we have $$d(p, x) = d(a_1, y) = d(a_2, y) = d(a_3, y).$$ However, since $a_1, a_2, a_3$ are distinct points that are collinear, there is no point $y$ that is equidistant from all three, leading to a contradiction. Hence, no line $\ell_{x, y}$ is in $R$. 
    
    Consider the horizontal line $S(\lambda', \lambda)$, where $\lambda'$ is some oriented line containing $p$ and not parallel to $\lambda$. For each $a_i$, let $g_i$ be the rigid motion obtained by first translating the plane to map $p$ to $a_i$, and then rotating the plane around $a_i$ to map $\lambda'$ to $\lambda$. Clearly, $g_i\in \ell_{p, a_i}$ and $g_i \in S(\lambda', \lambda)$. In other words, $S(\lambda', \lambda)$ intersect all three lines, and is contained $\Psi(\ell_{p, a_1}, \ell_{p, a_2}, \ell_{p, a_3})$. Hence, the set of lines \begin{equation}\label{eqn:paraboloid_ruling}
        \{S(\lambda', \lambda) \ |\ \mbox{$\lambda'$ is an oriented line containing $p$, not parallel to $\lambda$}\}
    \end{equation}
    is contained in the first ruling of the regulus.  
    
    Repeating the same argument for any $x_1, x_2, x_3 \in \lambda$, one can see that $\Psi(\ell_{p, x_1}, \ell_{p,x_2},\ell_{p,x_3}) = \Psi(\ell_{p,a_1}, \ell_{p,a_2},\ell_{p,a_3})$. By Theorem \ref{thm:r3_bezout}, the two triples of lines define the regulus $R$. 
    By Lemma \ref{lem:gk_regulus}, this ruling lies in $\curves_p^1$. Moreover, by Corollary \ref{corr:regulus_ruling_line}, any line in this ruling intersects all but a finite number of lines in the other ruling. Suppose $\ell_{p,v}$ intersects $S(\lambda', \lambda)$, where $\lambda'$ contains $p$. Let $g = \ell_{p, v} \cap S(\lambda', \lambda)$. Since $g$ is a rigid motion, we have $$d(v, \lambda) = d(p, \lambda') = 0.$$ Therefore, this ruling is exactly $\{\ell_{p, a}\}_{a \in \lambda}$. 
    
    Finally, we show that no other line lies in the ruling containing \eqref{eqn:paraboloid_ruling}. We already showed above that no line of the form $\ell_{x, y}$ can intersect three distinct lines $\ell_{p, a_1}, \ell_{p, a_2}, \ell_{p, a_3}$ where $a_i \in \lambda$. That is, it remains to consider only horizontal lines. Suppose $\lambda'' \neq \lambda$ is a line that does not contain $p$. For any $g \in S(\lambda'', \lambda)$, since $g$ is a bijection between $\lambda''$ and $\lambda$, $g(p) \notin \lambda$. In other words, $S(\lambda'', \lambda)$ cannot be transversal to any $\ell_{p, x}$ where $x \in \lambda$. Hence, this ruling cannot contain any line not in \eqref{eqn:paraboloid_ruling}. 
\end{proof}

\begin{remark}
    Since all the lines in one ruling of the above regulus are parallel to the $xy$-plane (recall that $S(\lambda', \lambda)$ is horizontal), the above regulus is a hyperbolic paraboloid.
\end{remark}

\begin{remark}
    As an example, consider the construction in Figure \ref{fig:orthogonal_example}. The corresponding set of lines $\lines$ that we obtain through the modified ESGK reduction results in $2m$ reguli of the form described in Proposition \ref{prop:regulus_by_line}, which contain $n$ lines each. 
\end{remark}

    We are now ready to prove Lemma \ref{lem:regulus_bound_main}. We restate it here for convenience. 

    \begingroup
    \def\thetheorem{\ref{lem:regulus_bound_main}}
    \begin{lemma}
    For $\lines$ as defined in \eqref{eqn:ESGK_lines_defn}, at least one of the following holds.
    \begin{enumerate}[label=(\roman*), itemsep = \enumsep]
        \item $D(\pts, \qts) = \Omega(\sqrt{mn})$.
        \item Every regulus in $\RR^3$ contains $O(\sqrt{mn})$ lines of $\lines$.
    \end{enumerate}
    \end{lemma}
    \addtocounter{theorem}{-1}
    \endgroup
    
\begin{proof}
    Consider some regulus $R$. In the case where every family $\lines_p^i$ only has at most four lines that lie in the regulus, there are $O(m)$ lines in $R$, and we are done.
    
    Consider some family $\lines_p^1$ and a regulus $R$ such that $R$ contains at least five lines from this family. Consider three distinct lines, $\ell_{p, a}, \ell_{p, b}, \ell_{p, c}$ that lie on the same ruling of $R$. Note that $R = R(\ell_{p, a}, \ell_{p, b}, \ell_{p, c})$ for any three distinct points in $\RR^2$, they are either collinear or lie on a common circle. First, assume that $a, b, c$ lie on some circle $C(q, r)$ for some $q \in \RR^2$. By proposition \ref{prop:regulus_by_circle}, the ruling containing $\ell_{p, a}, \ell_{p, b}, \ell_{p, c}$ consists of lines of the form $\{\ell_{p, x}\}_{x \in C(q, r)}.$ Hence, the number of lines in this ruling corresponds to the number of points in $\qts$ that lie on $C(q, r)$. Similarly, if $a, b$ and $c$ lie on a line $\lambda$, Proposition \ref{prop:regulus_by_line} implies that this ruling consists of lines of the form $\{\ell_{p, x}\}_{x \in \lambda}$. Hence, the number of lines in this ruling corresponds to the number of points in $\qts$ that lie on $\lambda$. A symmetric argument can be applied in the case where some regulus contains at least five lines from a family $\lines_p^2$. 
    
    We assume that $D(\pts, \qts) = O(\sqrt{mn})$, since otherwise we are done. By Lemma \ref{lem:curve_bipartite_bound}, every circle or line contains $O(\sqrt{mn})$ points. Therefore, for either of the above cases, we have that $R$ contains $O(\sqrt{mn})$ lines, as desired. 
\end{proof}


\section*{Acknowledgements}
The author would like to sincerely thank Adam Sheffer for introducing her to this problem, for his patient and enthusiastic guidance throughout this research endeavour, and for reviewing drafts of this work and giving invaluable feedback. The author would also like to thank Frank de Zeeuw and Pablo Soberon for helpful discussions.

\nocite{*}
  \bibliography{ms}

\begin{thebibliography}{ACNS82}

\bibitem[ACNS82]{CrossingLemma}
M.~Ajtai, V.~Chv{\' a}tal, M.~Newborn, and E.~Szemer{\'e}di.
\newblock Crossing-free subgraphs.
\newblock {\em Annals of Discrete Mathematics}, 12:9--12, 1982.

\bibitem[BES19]{BardwellSheffer}
S.~Bardwell-Evans and A.~Sheffer.
\newblock A reduction for the distinct distances problem in $\mathbb{R}^d$.
\newblock {\em J. Combinat. Theory A}, 166:171--224, 2019.

\bibitem[BMP05]{ResearchProblems}
P.~Brass, W.~Moser, and J.~Pach.
\newblock {\em Research Problems in Discrete Geometry}.
\newblock Springer-Verlag, New York, 2005.

\bibitem[Chu84]{Chung}
F.~R.~K. Chung.
\newblock The number of different distances determined by $n$ points in the
  plane.
\newblock {\em J. Combin. Theory Ser. A}, 36:342--354, 1984.

\bibitem[CLO15]{AlgGeom}
D.~Cox, J.~Little, and D.~O'Shea.
\newblock {\em Ideals, Varieties, and Algorithms: An Introduction to
  Computational Algebraic Geometry and Commutative Algebra, 4th Edition}.
\newblock Springer-Verlag, Heidelberg, 2015.

\bibitem[CST92]{Chungetal}
F.~R.~K. Chung, E~Szemer{\'e}di, and W.T. Trotter.
\newblock The number of different distances determined by a set of points in
  the euclidean plane.
\newblock {\em Discrete Comput. Geom.}, 7:1--11, 1992.

\bibitem[dZP17]{BDDCurves}
F.~de~Zeeuw and J\'{a}nos Pach.
\newblock Distinct distances on algebraic curves in the plane.
\newblock {\em Combinatorics, Probability and Computing}, 26:2017, 2017.

\bibitem[Ele95]{CircleGrids}
G.~Elekes.
\newblock Circle grids and bipartite graphs of distances.
\newblock {\em Combinatorica}, 15:167--174, 1995.

\bibitem[EP71]{ErdosPurdy}
P.~Erd{\H o}s and G.~Purdy.
\newblock Some extremal problems in geometry.
\newblock {\em J. Combinat. Theory}, 10:246–252, 1971.

\bibitem[Erd46]{ErdosOriginal}
P.~Erd{\H o}s.
\newblock On sets of distances of $n$ points.
\newblock {\em Amer. Math. Monthly}, 53:248--250, 1946.

\bibitem[Erd95]{Er01}
P.~Erd{\H o}s.
\newblock A selection of problems and results in combinatorics.
\newblock In {\em Recent trends in combinatorics (Matrahaza, 1995)}, pages
  1--6. Cambridge Univ. Press, Cambridge, 1995.

\bibitem[ES10]{ElekesSharir}
G.~Elekes and M.~Sharir.
\newblock Incidences in three dimensions and distinct distances in the plane.
\newblock {\em Proceedings 26th ACM Symposium on Computational Geometry}, pages
  413--422, 2010.

\bibitem[GIS11]{AMS}
J.~Garibaldi, A.~Iosevich, and S.~Senger.
\newblock {\em The Erd{\H o}s Distance Problem}.
\newblock Amer. Math. Soc. Press, Providence, RI, 2011.

\bibitem[GK10]{Joints}
L.~Guth and N.~Katz.
\newblock Algebraic methods in discrete analogs of the {K}akeya problem.
\newblock {\em Advances in Mathematics}, 255:2828--2839, 2010.

\bibitem[GK15]{GuthKatz}
L.~Guth and N.~H. Katz.
\newblock On the {E}rd{\H o}s distinct distances problem in the plane.
\newblock {\em Annals of Mathematics}, 181:155--190, 2015.

\bibitem[Gut16]{GuthBook}
L.~Guth.
\newblock {\em Polynomial Methods in Combinatorics}.
\newblock Amer. Math. Soc. Press, 2016.

\bibitem[Har92]{Complexity}
J.~Harris.
\newblock {\em , Algebraic geometry: a first course}.
\newblock Springer, New York, 1992.

\bibitem[KT04]{KatzTardos}
N.~H. Katz and G.~Tardos.
\newblock A new entropy inequality for the {E}rd{\H o}s distance problem.
\newblock {\em Contemporary Mathematics}, 342:119--126, 2004.

\bibitem[Lei83]{CrossingLemma2}
F.~T. Leighton.
\newblock {\em Complexity Issues in VLSI}.
\newblock M.I.T. Press, Cambridge, MA, 1983.

\bibitem[Mos52]{Moser}
L.~Moser.
\newblock On the different distances determined by $n$ points.
\newblock {\em Amer. Math. Monthly}, 59:85--91, 1952.

\bibitem[She]{AdamTextbook}
Adam Sheffer.
\newblock Polynomial methods and incidence theory.
\newblock available at
  \url{http://faculty.baruch.cuny.edu/ASheffer/000book.pdf}.
\newblock Accessed: 2019-10-05.

\bibitem[She18]{Survey}
A.~Sheffer.
\newblock Distinct distances: Open problems and current bounds,.
\newblock {\em arXiv:1406.1949v3}, 2018.

\bibitem[SSS13]{BDDLines}
M.~Sharir, A.~Sheffer, and J.~Solymosi.
\newblock Distinct distances on lines.
\newblock {\em J. Combinat. Theory A.}, 120:1732--1736, 2013.

\bibitem[ST83]{SzeTrotter}
E.~Szemer{\'e}di and W.~T. Trotter.
\newblock Extremal problems in discrete geometry.
\newblock {\em Combinatorica}, 3:381–392, 1983.

\bibitem[ST01]{SolymosiToth}
J.~Solymosi and G.~T{\'o}th.
\newblock Distinct distances in the plane.
\newblock {\em Discrete Comput. Geom}, 25:629--634, 2001.

\bibitem[Sz{\'e}87]{InclusionExclusion}
L.~Sz{\'e}kely.
\newblock Inclusion-exclusion formulae without higher terms.
\newblock {\em Ars Combinatoria}, 23B:7--20, 1987.

\bibitem[Sz{\'e}97]{Szekely}
L.~Sz{\'e}kely.
\newblock Crossing numbers and hard {E}rd{\H o}s problems in discrete geometry.
\newblock {\em Combin. Probab. Comput.}, 6:353--358, 1997.

\bibitem[Tar03]{Tardos}
G.~Tardos.
\newblock On distinct sums and distinct distances.
\newblock {\em Adv. Math.}, 180:275--289, 2003.

\bibitem[Whi57]{Whitney}
H.~Whitney.
\newblock Elementary structure of real algebraic varieties.
\newblock {\em Annals of Math}, 66:546--556, 1957.

\end{thebibliography}
  \bibliographystyle{alpha}
  
\end{document}